\documentclass{amsart}
\usepackage{hyperref,threeparttable}
\usepackage[margin=1.5cm]{geometry}
\newcommand{\Sym}{\textrm{Sym}}
\newcommand{\fix}{\textrm{fix}}
\newtheorem{prop}{Proposition}[section]
\newtheorem{thm}[prop]{Theorem}
\newtheorem{conj}[prop]{Conjecture}

\newtheorem{example}[prop]{Example}

\newtheorem{lem}[prop]{Lemma}
\newtheorem{defn}[prop]{Definition}
\theoremstyle{definition}

\numberwithin{equation}{section}
\newcommand{\symme}{\mathrm{Sym}}
\newcommand{\Fix}{\mathrm{Fix}}
\def\cent#1#2{{\bf C}_{{#1}}{{(#2)}}}
\def\nor#1#2{{\bf N}_{{#1}}{{(#2)}}}
\def\Zent#1{{\bf Z}{{(#1)}}}
\begin{document}
\title{Cherlin's conjecture for sporadic simple groups}  

\author{Francesca Dalla Volta}
\address{Dipartimento di Matematica e Applicazioni, University of Milano-Bicocca, Via Cozzi 55, 20125 Milano, Italy} 
\email{francesca.dallavolta@unimib.it}

\author{Nick Gill}
\address{ Department of Mathematics, University of South Wales, Treforest, CF37 1DL, U.K.}
\email{nick.gill@southwales.ac.uk}

\author{Pablo Spiga}
\address{Dipartimento di Matematica e Applicazioni, University of Milano-Bicocca, Via Cozzi 55, 20125 Milano, Italy} 
\email{pablo.spiga@unimib.it}

\begin{abstract}
 We prove Cherlin's conjecture, concerning binary primitive permutation groups, for those groups with socle isomorphic to a sporadic simple group. 
\end{abstract}
\maketitle
\section{Introduction}\label{s: introduction}

In this paper we consider the following conjecture which is due to Cherlin, and which was given first in \cite{cherlin1}:

\begin{conj}\label{conj: cherlin} A finite primitive binary permutation group must be one of the following:
\begin{enumerate}
\item a Symmetric group $\Sym (n)$ acting naturally on $n$ elements;
\item a cyclic group of prime order acting regularly on itself;
\item an affine orthogonal group $V\cdot O(V)$ with $V$ a vector space over a finite field equipped with an anisotropic quadratic form acting on itself by translation, with complement the full orthogonal group $O(V)$.
\end{enumerate}
\end{conj} 

Thanks to work of Cherlin himself \cite{cherlin2}, and of Wiscons \cite{wiscons}, Conjecture~\ref{conj: cherlin} has been reduced to a statement about almost simple groups. In particular, to prove Conjecture~\ref{conj: cherlin} it would be sufficient to prove the following statement.

\begin{conj}\label{conj: cherlin2}
 If $G$ is a binary almost simple primitive permutation group on the set $\Omega$, then $G=\symme(\Omega)$. 
\end{conj}

In this paper, we prove this conjecture for almost simple groups with sporadic socle. Formally, our main result is the following:

\begin{thm}\label{t: sporadic}
 Let $G$ be an almost simple primitive permutation group with socle isomorphic to a sporadic simple group. Then $G$ is not binary.
\end{thm}

Note that we include the group ${^2F_4}(2)'$ in the list of sporadic groups -- this group is sometimes considered ``the $27^{\rm th}$ sporadic group'' -- so Theorem~\ref{t: sporadic} applies to this group too.

The terminology of Theorem~\ref{t: sporadic} and the preceding conjectures is all fairly standard in the world of group theory, with the possible exception of the word ``binary''. Roughly speaking an action is ``binary'' if the induced action on $\ell$-tuples can be deduced from the induced action on pairs (for any integer $\ell>2$); a formal definition of a ``binary permutation group'' is given below in \S\ref{s: preliminaries1}.

\subsection{Context and methods}

We will not spend much time here trying to motivate the study of ``binary permutation groups''. As will be clear on reading the definition of ``binary'' in \S\ref{s: preliminaries1}, this notion is a particular instance of the more general concept of ``arity'' or ``relational complexity''. These notions, which we define below in group theoretic terms, can also be formulated from a model theoretic point of view where they are best understood as properties of ``relational structures''. These connections, which run very deep, are explored at length in \cite{cherlin1}, to which we refer the interested reader.

Theorem~\ref{t: sporadic} settles Conjecture~\ref{conj: cherlin2} for one of the families given by the Classification of Finite Simple Groups. It is the third recent result in this direction: Conjecture~\ref{conj: cherlin2} has also been settled for groups with alternating socle \cite{gs_binary}, and for groups with socle a rank $1$ group of Lie type \cite{ghs_binary}. Work is ongoing for the groups that remain (groups with socle a group of Lie type of rank at least $2$) \cite{gls_binary}.

Our proof of Theorem~\ref{t: sporadic} builds on ideas developed in \cite{gs_binary} and \cite{ghs_binary}, in particular the notion of a ``strongly non-binary action''. In addition to this known approach, we also make use of a number of new lemmas -- we mention, in particular, Lemma~\ref{l: characters}, which connects the ``binariness'' of an action to a bound on the number of orbits in the induced action on $\ell$-tuples. These lemmas are gathered together in \S\ref{s: preliminaries1}.

In addition to these new lemmas, though, the current paper is very focused on adapting known facts about binary actions to create computational tests that can be applied using a computer algebra package like {\tt GAP} \cite{GAP4} or {\tt magma} \cite{magma}. This process of developing tests is explained in great detail in \S\ref{s: preliminaries2}.

In the final two sections we describe the outcome of these computations. In \S\ref{s: nonmonster} we are able to give a proof of Theorem~\ref{t: sporadic} for all of the sporadic groups barring the monster. In \S\ref{s: monster} we give a proof of Theorem~\ref{t: sporadic} for the monster. The sheer size of the monster means that some of the computational procedures that we exploit for the other groups are no longer available to us, and so our methods need to be refined to deal with this special case.

\subsection{Acknowledgments}
At a crucial juncture in our work on Theorem~\ref{t: sporadic}, we needed access to greater computational power -- this need was met by Tim Dokchitser who patiently ran and re-ran various scripts on the University of Bristol {\tt magma} cluster. We are very grateful to Tim -- without his help we would have struggled to complete this work. We are also grateful to an anonymous referee for a number of helpful comments and suggestions.

\section{Definitions and lemmas}\label{s: preliminaries1}

Throughout this section $G$ is a finite group acting (not necessarily faithfully) on a set $\Omega$ of cardinality $t$. Given a subset $\Lambda$ of $\Omega$, we write $G_\Lambda:=\{g\in G\mid \lambda^g\in\Lambda,\forall \lambda\in \Lambda\}$ for the set-wise stabilizer of $\Lambda$, $G_{(\Lambda)}:=\{g\in G\mid \lambda^g=\lambda, \forall\lambda\in \Lambda\}$ for the point-wise stabilizer of $\Lambda$, and $G^\Lambda$ for the permutation group induced on $\Lambda$ by the action of $G_\Lambda$. In particular, $G^\Lambda\cong G_\Lambda/G_{(\Lambda)}$.

Given a positive integer $r$, the group $G$ is called \textit{$r$-subtuple complete} with respect to the
pair of $n$-tuples $I, J \in \Omega^n$, if it contains elements that
map every subtuple of length $r$ in $I$ to the corresponding subtuple in
$J$ i.e. $$\textrm{for every } \{k_1, k_2, \dots, k_r\}\subseteq\{ 1,
\ldots, n\}, \textrm{ there exists } h \in G \textrm{ with }I_{k_i}^h=J_{k_i}, \textrm{ for every }i \in\{
1, \ldots, r\}.$$ Here $I_k$ denotes the $k^{\text{th}}$ element of tuple
$I$ and $I_k^g$ denotes the image of $I_k$ under the action of $g$.
Note that $n$-subtuple completeness simply requires the existence of
an element of $G$ mapping $I$ to $J$.

\begin{defn}{\rm
The action of $G$ is said to be of {\it arity $r$} if, for all
$n\in\mathbb{N}$ with $n\geq r$ and for all $n$-tuples $I, J \in \Omega^n$, $r$-subtuple
completeness (with respect to $I$ and $J$) implies $n$-subtuple completeness (with respect to $I$ and $J$). Note that in the literature the concept of ``arity'' is also known by the name ``relational complexity''. 

When the action of $G$ has arity 2, we say that the action of $G$ is {\it binary}. If $G$ is given to us as a permutation group, then we say that $G$ is a \emph{binary permutation group}. 

  }
  \end{defn}

A pair $(I,J)$ of $n$-tuples of $\Omega$ is called a {\it non-binary witness for the action of $G$ on $\Omega$}, if $G$ is $2$-subtuple complete with respect to $I$ and $J$, but not $n$-subtuple complete, that is, $I$ and $J$ are not $G$-conjugate.
To show that the action of $G$ on $\Omega$ is non-binary it is sufficient to find a non-binary witness $(I,J)$. 

We now recall some useful definitions introduced in~\cite{ghs_binary}. We say that the action of $G$ on $\Omega$ is \emph{strongly non-binary} if there exists a non-binary witness $(I,J)$ such that
\begin{itemize}
 \item $I$ and $J$ are $t$-tuples where $|\Omega|=t$;
 \item the entries of $I$ and $J$ comprise all the elements of $\Omega$.
\end{itemize}

We give a standard example, taken from~\cite{ghs_binary},  showing how strongly non-binary actions can arise.

\begin{example}\label{ex: snba2}{\rm
Let $G$ be a subgroup of  $\Sym(\Omega)$, let $g_1, g_2,\ldots,g_r$ be elements of $G$, and let $\tau,\eta_1,\ldots,\eta_r$ be elements of $\Sym(\Omega)$ with
\[
 g_1=\tau\eta_1,\,\,g_2=\tau\eta_2,\,\,\ldots,\,\,g_r=\tau\eta_r.
\]
Suppose that, for every $i\in \{1,\ldots,r\}$, the support of $\tau$ is disjoint from the support of $\eta_i$; moreover, suppose  that, for each $\omega\in\Omega$, there exists $i\in\{1,\ldots,r\}$ (which may depend upon $\omega$) with $\omega^{\eta_i}=\omega$. Suppose, in addition, $\tau\notin G$.
Now, writing $\Omega=\{\omega_1,\dots, \omega_t\}$, observe that
 \[
  ((\omega_1,\omega_2,\dots, \omega_t), (\omega_1^{\tau},\omega_2^{\tau}, \ldots,\omega_t^{\tau}))
 \]
is a non-binary witness. Thus the action of $G$ on $\Omega$ is strongly non-binary.}
\end{example}

The following lemma, taken from~\cite{ghs_binary}, shows a crucial property of the notion of strongly non-binary action: it allows one to argue ``inductively'' on set-stabilizers (see also Lemma~\ref{l: again0}). 

\begin{lem}\label{l: again12}
Let $\Omega$ be a  $G$-set and let $\Lambda \subseteq \Omega$. If $G^\Lambda$ is strongly non-binary, then $G$ is not binary in its action on $\Omega$.
\end{lem}
\begin{proof}Write $\Lambda:=\{\lambda_1,\ldots,\lambda_\ell\}$ and assume that $G^\Lambda$ is strongly non-binary. Then there exists $\sigma\in \Sym(\ell)$ with $I:=(\lambda_1,\lambda_2,\ldots,\lambda_\ell)$ and  $J:=(\lambda_{1^\sigma},\lambda_{2^\sigma},\ldots,\lambda_{\ell^\sigma})$ a non-binary witness for the action of $G_\Lambda$ on $\Lambda$. Now, observe that $(I,J)$ is also a non-binary witness for the action of $G$ on $\Omega$ because any (putative) element $g$ of $G$ mapping $I$ to $J$ fixes $\Lambda$ set-wise and hence $g\in G_\Lambda$.  
\end{proof}

Next we need an observation, made first in \cite{ghs_binary}, that the existence of a strongly non-binary witness is related to the classic concept of $2$-\emph{closure} introduced by Wielandt~\cite{Wielandt}: given a permutation group $G$ on $\Omega$, the \emph{$2$-closure of $G$} is the set $$G^{(2)}:=\{\sigma\in \Sym(\Omega)\mid \forall (\omega_1,\omega_2)\in \Omega\times \Omega, \textrm{there exists }g_{\omega_1\omega_2}\in G \textrm{ with }\omega_1^\sigma=\omega_1^{g_{\omega_1\omega_2}}, \omega_2^\sigma=\omega_2^{g_{\omega_1\omega_2}}\},$$
that is, $G^{(2)}$ is the largest subgroup of $\Sym(\Omega)$ having the same orbitals as $G$. The group $G$ is said to be $2$-\emph{closed} if and only if $G=G^{(2)}$. 

\begin{lem}\label{l: fedup}
Let $G$ be a permutation group on $\Omega$. Then $G$ is strongly non-binary if and only if $G$ is not $2$-closed.
\end{lem}
\begin{proof}
Write $\Omega:=\{\omega_1,\ldots,\omega_t\}$. If $G$ is not $2$-closed, then there exists $\sigma\in G^{(2)}\setminus G$. Set $I:=(\omega_1,\ldots,\omega_t)$ and $J:=I^\sigma=(\omega_1^\sigma,\ldots,\omega_t^\sigma)$; observe that $I$ and $J$ are $2$-subtuple complete (because $\sigma\in G^{(2)}$) and are not $G$-conjugate (because $\sigma\notin G$). Thus $(I,J)$ is a strongly non-binary witness. The converse is similar.
\end{proof}

Our next two lemmas make use of Lemma~\ref{l: again12} and Example~\ref{ex: snba2} to yield easy criteria for showing that a permutation group is not binary.

\begin{lem}\label{l: M2}
Let $G$ be a transitive permutation group on $\Omega$, let $\alpha\in \Omega$ and let $p$ be a prime with $p$ dividing both $|\Omega|$ and $|G_\alpha|$ and with $p^2$ not dividing $|G_\alpha|$. Suppose that $G$ contains an elementary abelian $p$-subgroup $V=\langle g,h\rangle$ with $g\in G_\alpha$, with $h$ and $gh$ conjugate to $g$ via $G$. Then $G$ is not binary. 
\end{lem}
\begin{proof}
Let $g\in G_\alpha$ and let $h\in g^G$ with $\langle g,h\rangle$ an elementary abelian $p$-subgroup of $G$ of order $p^2$ with $gh$ also conjugate to $g$ via $G$. In particular, $h=g^x$, for some $x\in G$. Write $\alpha_0:=\alpha$ and  $\alpha_{p}:=\alpha^x$.

Since $g\in G_{\alpha_0}$ and $h\in G_{\alpha_{p}}$ commute, $\alpha_0^{h^i}$ is fixed by $g$ and $\alpha_{p}^{g^i}$ is fixed by $h$, for every $i$. Write $\alpha_i:=\alpha_0^{h^i}$ and $\alpha_{p+i}:=\alpha_p^{g^i}$, for every $i\in \{0,\ldots,p-1\}$. Moreover, $g$ acts as a $p$-cycle on $\{\alpha_p,\ldots,\alpha_{2p-1}\}$ and $h$ acts as a $p$-cycle on $\{\alpha_0,\ldots,\alpha_{p-1}\}$.

Since $gh$ is conjugate to $g$ via an element of $G$, there exists $y\in G$ with $gh=g^y$. Write $\alpha_{2p}=\alpha^{y}$. Observe that $gh$ fixes $(\alpha_{2p})^{g^{-i}}=\alpha_{2p}^{h^i}$ for every $i$. Write $\alpha_{2p+i}:=\alpha_{2p}^{g^i}$, for every $i\in \{0,\ldots,p-1\}$. Thus $g$ and $h$ act as inverse $p$-cycles on $\{\alpha_{2p},\ldots,\alpha_{3p-1}\}$. 

Write $\Lambda:=\{\alpha_0,\ldots,\alpha_{3p-1}\}$. We have
\begin{align*}
g^\Lambda&=(\alpha_{p},\ldots,\alpha_{2p-1})(\alpha_{3p-1},\ldots,\alpha_{2p}),\\
h^\Lambda&=(\alpha_{0},\ldots,\alpha_{p-1})(\alpha_{2p},\ldots,\alpha_{3p-1}),\\
(gh)^\Lambda&=(\alpha_{0},\ldots,\alpha_{p-1})(\alpha_{p},\ldots,\alpha_{2p-1}).\\
\end{align*}

If $G^\Lambda$ is  strongly non-binary, then $G$ is not binary by Lemma~\ref{l: again12}. Assume that $G^\Lambda$ is not strongly non-binary. Then, in view of Example~\ref{ex: snba2}, there exists $f\in G$ with 
$f^\Lambda=(\alpha_p,\ldots,\alpha_{2p-1}).$
This is a contradiction, because by hypothesis $|G_\alpha|$ is not divisible by $p^2$ but $\langle g,f\rangle$ has order divisible by $p^2$ and fixes $\alpha_0=\alpha$.
\end{proof}

\begin{lem}\label{l: added}
Let $G$ be a permutation group on $\Omega$ and suppose that $g$ and $h$ are elements of $G$ of order $p$ where $p$ is a prime such that $g$, $h$ and $gh^{-1}$ are all $G$-conjugate. Suppose that $V=\langle g, h\rangle$ is elementary-abelian of order $p^2$. Suppose, 
finally, that $G$ does not contain any elements of order $p$ that fix more points of $\Omega$ than $g$. If $|\Fix(V)|<|\Fix(g)|$, then $G$ is not binary.
\end{lem}

We remark that there are well-known formulae that we can use to calculate $\Fix(V)$ and $|\Fix(g)|$ when $G$ is transitive (see for instance~\cite[Lemma~$2.5$]{LiebeckSaxl}). Suppose that $M$ is the stabilizer of a point in $\Omega$; then we have
\begin{equation}\label{e: fora}
 |\Fix_\Omega(g)| = \frac{|\Omega|\cdot |M\cap g^G|}{|g^G|},\qquad
 |\Fix_\Omega(V)| = \frac{|\Omega|\cdot |\{V^g\mid g\in G,V^g\le M\}|}{|V^G|}.
\end{equation}

\begin{proof}
We let
\[
 \Lambda:=\Fix(g)\cup\Fix(h)\cup\Fix(gh^{-1}).
\]
Observe, first, that $\Lambda$, $\Fix(g)$, $\Fix(h)$ and $\Fix(gh^{-1})$ are $g$-invariant and $h$-invariant. Observe, second, that 
\[
\Fix(g)\cap\Fix(h)=\Fix(g)\cap\Fix(gh^{-1})=\Fix(h)\cap\Fix(gh^{-1})=\Fix(V). 
\]

Write $\tau_1$ for the permutation induced by $g$ on $\Fix(gh^{-1})$, $\tau_2$ for the permutation induced by $g$ on $\Fix(h)$, and $\tau_3$ for the permutation induced by $h$ on $\Fix(g)$ (observe that $\tau_i$'s are non trivial as $gh^{-1}$, $h$ and $g$ are conjugate).
Since $|\Fix(V)|<|\Fix(g)|$, we conclude that $\tau_1,\tau_2$ and $\tau_3$ are disjoint non-trivial permutations. What is more, $g$ induces the permutation $\tau_1\tau_2$ on $\Lambda$, while $h$ induces the permutation $\tau_1\tau_3$ on $\Lambda$. 

In view of Example~\ref{ex: snba2}, $G^\Lambda$ is strongly non-binary provided there is no element $f\in G_\Lambda$ that induces the permutation $\tau_1$. Arguing  by contradiction, if such an element $f$ exists, then $f$ has order divisible by $p$ and $f^{o(f)/p}$ is a $p$-element fixing more points than $g$, which is a contradiction.
Thus $G^\Lambda$ is strongly non-binary and $G$ is not binary by Lemma~\ref{l: again12}.
\end{proof}

For the rest of this section we assume that $G$ is transitive. Given $\ell\in\mathbb{N}\setminus\{0\}$, we denote by $\Omega^{(\ell)}$ the subset of the Cartesian product $\Omega^\ell$ consisting of the $\ell$-tuples $(\omega_1,\ldots,\omega_\ell)$ with $\omega_i\ne \omega_j$, for every two distinct elements $i,j\in \{1,\ldots,\ell\}$. We denote by $r_\ell(G)$  the number of orbits of $G$ on $\Omega^{(\ell)}$.

Let $\pi:G\to\mathbb{N}$ be the permutation character of $G$, that is, $\pi(g)=\fix_\Omega(g)$ where $\fix_{\Omega}(g)$ is the cardinality of the fixed point set $\Fix_{\Omega}(g):=\{\omega\in \Omega\mid \omega^g=\omega\}$ of $g$. From the Orbit Counting Lemma, we have
\begin{align*}
r_\ell(G)&=\frac{1}{|G|}\sum_{g\in G}\fix_\Omega(g)(\fix_\Omega(g)-1)\cdots (\fix_\Omega-(\ell-1))\\
&=\langle \pi(\pi-1)\cdots (\pi-(\ell-1)),1\rangle_G,
\end{align*}
where $1$ is the principal character of $G$ and $\langle \cdot,\cdot\rangle_G$ is the natural Hermitian product on the space of $\mathbb{C}$-class functions of $G$.

\begin{lem}\label{l: characters}
If $G$ is transitive and binary, then $r_\ell(G)\le r_2(G)^{\ell(\ell-1)/2}$ for each $\ell\in\mathbb{N}$.
\end{lem}

Note that this lemma is, in effect, an immediate consequence of the fact that, for a binary action, the orbits on pairs ``determine'' orbits on $\ell$-tuples. Thus, to uniquely determine the orbit of a particular $\ell$-tuple, it is enough to specify the orbits of all $\binom{\ell}{2}$ pairs making up the $\ell$-tuple.

\begin{proof}
We write $r_2:=r_2(G)$ and $r_\ell:=r_\ell(G)$ and we assume that $r_\ell>r_2^{(\ell-1)\ell/2}$ for some $\ell\in\mathbb{N}$. Clearly, $\ell> 2$.

 Let $$(\omega_{1,1},\ldots,\omega_{1,\ell}),\ldots,(\omega_{r_\ell,1},\ldots,\omega_{r_\ell,\ell})$$ be a family of representatives for the $G$-orbits on $\Omega^{(\ell)}$. From the pigeon-hole principle, at least $r_\ell/r_2$ of these elements have the first two coordinates in the same $G$-orbit. Formally, there exists $\kappa\in\mathbb{N}$ with $\kappa\ge r_\ell/r_2$ and a subset $\{i_1,\ldots,i_\kappa\}$ of $\{1,\ldots,r_\ell\}$ of cardinality $\kappa$ such that the $\kappa$ pairs
$$(\omega_{i_1,1},\omega_{i_1,2}),\ldots,(\omega_{i_\kappa,1},\omega_{i_\kappa,2})$$
are in the same $G$-orbit. By considering all possible pairs of coordinates, this argument can be easily generalized. Indeed, from the pigeon-hole principle, there exists $\kappa$ with $\kappa\ge r_\ell/r_2^{(\ell-1)\ell/2}>1$ and a subset $\{i_1,\ldots,i_\kappa\}$ of $\{1,\ldots,r_\ell\}$ of cardinality $\kappa$ such that, for each $1\le u<v\le \ell$, the $\kappa$ pairs
$$(\omega_{i_1,u},\omega_{i_1,v}),\ldots,(\omega_{i_\kappa,u},\omega_{i_\kappa,v})$$
are in the same $G$-orbit. In other words, the $\ell$-tuples
$$(\omega_{i_1,1},\ldots,\omega_{i_1,\ell}),\ldots,(\omega_{i_\kappa,1},\ldots,\omega_{i_\kappa,\ell})$$
are $2$-subtuple complete. Since $G$ is binary, these $\ell$-tuples must be in the same $G$-orbit, contradicting $\kappa>1$.
\end{proof}

Observe that when $r_2(G)=1$, that is, $G$ is $2$-transitive,  Lemma~\ref{l: characters} yields $r_\ell(G)=1$ for every $\ell\in \{2,\ldots,|\Omega|\}$. Therefore $G=\Sym(\Omega)$ is the only $2$-transitive binary group. 

\begin{lem}\label{l: again0}
Let $G$ be transitive, let $\alpha$ be a point of $\Omega$ and let $\Lambda\subseteq \Omega$ be a $G_\alpha$-orbit. If  $G$ is binary, then $G_\alpha^\Lambda$ is binary. In particular, if $g\in G$ and the action of $G_\alpha$ on the right cosets of $G_\alpha\cap G_\alpha^g$ in $G_\alpha$ is not binary, then $G$ is not binary.
\end{lem}
\begin{proof}Assume that $G$ is binary. Let $\ell\in\mathbb{N}$ and let $I:=(\lambda_1,\lambda_2,\ldots,\lambda_\ell)$ and $J:=(\lambda_1',\lambda_2',\ldots,\lambda_\ell')$ be two tuples in $\Lambda^\ell$ that are $2$-subtuple complete for the action of $G_\alpha$ on $\Lambda$. Clearly, $I_0:=(\alpha,\lambda_1,\lambda_2,\ldots,\lambda_\ell)$ and $J_0:=(\alpha,\lambda_1',\lambda_2',\ldots,\lambda'_\ell)$ are $2$-subtuple complete for the action of $G$ on $\Omega$. As $G$ is binary, $I_0$ and $J_0$ are in the same $G$-orbit; hence $I$ and $J$ are in the same $G_\alpha$-orbit. From this we deduce that $G_\alpha^\Lambda$ is binary. 

Suppose now that $g\in G$ and that the action of $G_\alpha$ on the right cosets of $G_\alpha\cap G_\alpha^g$ in $G_\alpha$ is not binary. Set
$\beta:=\alpha^g$ and $\Lambda:=\beta^{G_\alpha}$. Now $\Lambda$ is a $G_\alpha$-orbit contained in $\Omega\setminus \{\alpha\}$ and the action of $G_\alpha$ on $\Lambda$ is permutation isomorphic to the action of $G_\alpha$ on the right cosets of $G_\alpha\cap G_\beta=G_\alpha\cap G_\alpha^g$ in $G_\alpha$. Therefore, $G_\alpha^{\Lambda}$ is not binary and hence $G$ is not binary. 
\end{proof}

\section{On computation}\label{s: preliminaries2}

In this section we explain how to make use of the lemmas given in the previous section in a computational setting. The computational problem we are faced with is as follows: given a transitive action of a group $G$ on a set $\Omega$, we wish to show that the action is non-binary; in some cases we will require more, namely that the action is strongly non-binary. If the set $\Omega$ is small enough, then we can often exhibit $G$ as a permutation group in the computer algebra package {\tt magma} and compute explicitly; when $\Omega$ gets too large, then this may be infeasible and we may know only the isomorphism type of $G$ and the isomorphism type of a point-stabilizer.

\subsection{Test 1: using Lemma~\ref{l: characters}.}
In some cases, Lemma~\ref{l: characters} is very efficient for dealing with some primitive actions of almost simple groups $G$ with socle a sporadic simple group. In particular, whenever the permutation character of $G$ is available in~\texttt{GAP}~\cite{GAP4} or in~\texttt{magma}~\cite{magma}, we can simply check directly the inequality in Lemma~\ref{l: characters}. For instance, using this method it is easy to verify that each faithful primitive action of $M_{11}$ is non-binary. 

For practical purposes, it is worth mentioning that apart from 
\begin{itemize}
\item the Monster,
\item the action of the Baby Monster on the cosets of a maximal subgroup of type $(2^2\times F_4(2)):2$,
\end{itemize} 
each permutation character of each primitive permutation representation of an almost simple group with socle a sporadic simple group is available in \texttt{GAP} via the package ``The GAP character Table Library". Therefore, for the proof of Theorem~\ref{t: sporadic}, we can quickly and easily use Lemma~\ref{l: characters} except for the Monster. To give a rough idea of the time to perform this test, in the Baby Monster (except for the action on the cosets on a maximal subgroup of type $(2^2\times F_4(2)):2$), it takes less than two minutes to perform this checking.  (The permutation character of the Baby Monster $G$ on the cosets of a maximal subgroup $M$ of type $(2^2\times F_4(2)):2$ is missing from the \texttt{GAP} library because the conjugacy fusion of some of the elements of $M$ in $G$ remains a mystery: this information is vital for computing the permutation character.)

For reasons that will be more clear later, for the proof of Theorem~\ref{t: sporadic}, we need to prove the non-binariness of permutation groups $G\le \Sym(\Omega)$ that are not necessarily almost simple, let alone having socle a sporadic simple group. When $|\Omega|$ is relatively small (for practical purposes, here relatively small means at most $10^9$), we can afford to compute the permutation character and check the inequality in Lemma~\ref{l: characters}.

\subsection{Test 2: using Lemma~\ref{l: fedup}.}
By connecting the notion of strong-non-binariness to 2-closure, Lemma~\ref{l: fedup} yields an immediate computational dividend: there are built-in routines in \texttt{GAP}~\cite{GAP4} and \texttt{magma}~\cite{magma} to compute the $2$-closure of a permutation group.

Thus if $\Omega$ is small enough, say $|\Omega|\le 10^6$, then we can easily check whether or not the group $G$ is $2$-closed. Thus we can ascertain whether or not $G$ is strongly non-binary.

\subsection{Test 3: a direct analysis}
The next test we discuss is feasible once again provided $|\Omega|\le 10^6$. It simply tests whether or not $2$-subtuple-completeness implies $3$-subtuple completeness, and the procedure is as follows:

We fix $\alpha\in \Omega$, we compute the orbits of $G_\alpha$ on $\Omega\setminus\{\alpha\}$ and we select a set of representatives $\mathcal{O}$ for these orbits. Then, for each $\beta\in \mathcal{O}$, we compute the orbits of $G_{\alpha}\cap G_{\beta}$ on $\Omega\setminus\{\alpha,\beta\}$ and we select a set of representatives $\mathcal{O}_\beta$. Then, for each $\gamma\in \mathcal{O}_\beta$, we compute $\gamma^{G_\alpha}\cap \gamma^{G_\beta}$. Finally, for each $\gamma'\in \gamma^{G_\alpha}\cap \gamma^{G_\beta}$, we test whether the two triples $(\alpha,\beta,\gamma)$ and $(\alpha,\beta,\gamma')$ are $G$-conjugate. If the answer is ``no'', then $G$ is not binary because by construction $(\alpha,\beta,\gamma)$ and $(\alpha,\beta,\gamma')$ are $2$-subtuple complete. In particular, in this circumstance, we can break all the ``for loops'' and deduce that $G$ is not binary. 

If the answer is ``yes'', for every $\beta,\gamma,\gamma'$, then we cannot deduce that $G$ is binary, but we can keep track of these cases for a deeper analysis. We observe that, if the answer is ``yes'', for every $\beta,\gamma,\gamma'$, then $2$-subtuple completeness implies $3$-subtuple completeness.

\subsection{Test 4: using Lemma~\ref{l: again0}.}
The next test is particularly useful in cases where $\Omega$ is very large, since its computational complexity is independent of $|\Omega|$. Let us suppose that $G$ and its subgroup $M$ are stored in a library as abstract groups (or as matrix groups or as permutation groups). When $|G:M|$ is too large, it is impractical (and sometimes impossible) to construct $G$ as a permutation group on the coset space $\Omega:=G/M$ with point stabilizer $M$. However, using Lemma~\ref{l: again0}, we can still prove that $G$ acting on $\Omega$ is non-binary: all we need is $g\in G$ such that the action of $M$ on $M\cap M^g$ is non-binary. Now, for carefully chosen $g$, $|M:M\cap M^g|$ might be much smaller than $|G:M|$ and we can use one of the previous tests to ascertain whether or not $M$ in its action on $M/(M\cap M^g)$ is binary. 

\subsection{Test 5: a new lemma}

Our final test requires an extra lemma which we include here, rather than in the earlier section, as its computational aspect is somehow inherent in its very statement.

\begin{lem}\label{l: alot}
Let $G$ be a primitive group on a set $\Omega$, let $\alpha$ be a point of $\Omega$, let $M$ be the stabilizer of $\alpha$ in $G$ and let $d$ be an integer with $d\ge 2$. Suppose $M\ne 1$ and, for each transitive action of $M$ on a set $\Lambda$ satisfying:
\begin{enumerate}
\item $|\Lambda|>1$, and 
\item every composition factor of $M$ is isomorphic to some section of $M^\Lambda$, and
\item either $M_{(\Lambda)}=1$ or, given $\lambda\in \Lambda$, the stabilizer $M_\lambda$ has a normal subgroup $N$ with $N\ne M_{(\Lambda)}$ and $N\cong M_{(\Lambda)}$, and
\item $M$ is binary in its action on $\Lambda$,
\end{enumerate}
we have that $d$ divides $|\Lambda|$. Then either $d$ divides $|\Omega|-1$ or $G$ is not binary.
\end{lem}
\begin{proof}
Suppose that $G$ is binary. Since $\{\beta\in\Omega\mid \beta^m=\beta,\forall m\in M\}$ is a block of imprimitivity for $G$ and since $G$ is primitive, we obtain that either $M$ fixes each point of $\Omega$ or $\alpha$ is the only point fixed by $M$. The former possibility is excluded because $M\neq 1$ by hypothesis. Therefore $\alpha$ is the only point fixed by $M$. Let $\Lambda\subseteq\Omega\setminus\{\alpha\}$ be an $M$-orbit.  Thus $|\Lambda|>1$ and (1) holds. Since $G$ is a primitive group on $\Omega$, from~\cite[Theorem~3.2C]{dixon_mortimer}, we obtain that every composition factor of $M$ is isomorphic to some section of $M^\Lambda$ and hence (2) holds. From Lemma~\ref{l: again0}, the action of $M$ on $\Lambda$ is binary and hence (4) holds. Let now $\lambda\in\Lambda$ and consider the orbital graph $\Gamma:=(\alpha,\lambda)^G$. Observe that $\Gamma$ is connected because $G$ is primitive. Let $g\in G$ with $\alpha^g=\lambda$. Clearly, $\Lambda$ is the set of out-neighbors of $\alpha$ in $\Gamma$ and $\Lambda':=\Lambda^g$ is the set of out-neighbors of $\alpha^g=\lambda$ in $\Gamma$. Set $N:=(G_\lambda)_{(\Lambda')}$. Clearly, $(G_{\alpha})_{(\Lambda)}=M_{(\Lambda)}$ and $(G_{\alpha^g})_{(\Lambda^g)}=(G_\lambda)_{(\Lambda')}=N$ are isomorphic because they are $G$-conjugate via the element $g$. Moreover, $M_{(\Lambda)}=(G_\alpha)_{(\lambda^{G_\alpha})}$ is normalized by $G_\alpha$ and, similarly, $N$ is normalized by $G_\lambda$: therefore they are both normalized by $G_\alpha\cap G_\lambda=M\cap G_\lambda=M_\lambda$. If $M_{(\Lambda)}$ and $N$ are equal, an easy connectedness argument yields that $M_{(\Lambda)}=1$. Therefore (3) also holds.

Since the four hypotheses in the statement of this lemma hold for the action of $M=G_\alpha$ on its $G_\alpha$-orbit $\Lambda$, we get $d$ divides $|\Lambda|$. Since this argument does not depend on the $G_\alpha$-orbit $\Lambda\subseteq\Omega\setminus\{\alpha\}$, we obtain that $\Omega\setminus\{\alpha\}$ has cardinality divisible by $d$. Thus $|\Omega|-1$ is divisible by $d$.  
\end{proof}

When it comes to implementing Lemma~\ref{l: alot} on a computer, it is important to observe that we do {\bf not} need to construct the embedding of $M=G_\alpha$ in $G$; indeed we do not need the group $G$ stored in our computer at all. Instead we need only the index $|G:M|=|\Omega|$ and the abstract group $M$ (given as a group of matrices, or as a permutation group, or as a finitely presented group). 

Given $|\Omega|$ and $M$, we may choose a prime $p$ (typically $p=2$) with $p$ not dividing $|\Omega|-1$ and we construct all the transitive permutation representations of degree greater than $1$ and relatively prime to $p$ of $M$ satisfying (1),~(2) and~(3). If none of these permutation representations is binary (and we can use any of Test 1 to 4 to test this), we infer that every transitive permutation representation of $M$ of degree greater than $1$ satisfying (1),~(2),~(3) and~(4) has degree divisible by $p$. Now, from Lemma~\ref{l: alot}, we get that $G$ in its action on the set $M$ of right cosets of $M$ in $G$ is not binary because $p$ does not divide $|\Omega|-1$.

We give an explicit example to show how easily Lemma~\ref{l: alot} can be applied. The monster $G$ has  a maximal subgroup $M$ isomorphic to $\mathrm{PGL}_2(19)$. The index of $M$ in $G$ is 
$$118131202455338139749482442245864145761075200000000\sim 10^{50}$$
and we can easily observe that this number is even. After implementing Lemma~\ref{l: alot} in a computer, it takes the blink of an eye to prove that each permutation representation of $M$ of degree at least $1$ and odd is non-binary. Thus $G$ acting on the cosets of $M$ is non-binary. Observe that besides $|G:M|$ and the isomorphism class of $M$, no information about $G$ is needed.

\section{The non-monster groups}\label{s: nonmonster}

The centre-piece of this section is Table~\ref{t: table1}; it summarises the results of applying the tests described in the previous section to all almost simple groups with a sporadic socle, barring the monster.

Table~\ref{t: table1} consists of two columns: the first column lists all of the almost simple groups $G$ with socle a sporadic simple group (recall that we include Tits group $^{2}F_4(2)'$ in the list of sporadic groups). In the second column, we list all pairs $(M,\circ)$, where $M$ is a maximal subgroup of $G$ with the property that the action of $G$ on the set $G/M$ of right cosets of $M$ in $G$ satisfies Lemma~\ref{l: characters} (in other words, the action is not excluded by Test 1, and hence is a  potentially binary action). We use the ATLAS~\cite{ATLAS} notation for the group $M$. 

Now the symbol $\circ$ is either $\infty$ or a prime $p$ or ``?''.  We write $\circ=\infty$ if we have proved the non-binariness of $G$ in its action on $G/M$ using Tests 2 or 3; we write $\circ=p$ if we have proved the non-binariness of $G$ in its action on $G/M$ using Test 5 applied to the prime $p$; and we write $\circ=?$ if both methods have failed. The symbol ``$-$" in the second column means that each primitive permutation representation of $G$ is not binary via Lemma~\ref{l: characters} (Test~1).


\begin{table}[!ht]
\begin{tabular}{|c|c|}\hline
Group &  Outcome of tests \\\hline
$M_{11}$&$-$\\
$M_{12}$&$-$\\
$M_{12}.2$&$-$\\
$M_{22}$&$-$\\
$M_{22}.2$&$-$\\
$M_{23}$&$-$\\
$M_{24}$&$(L_2(7),\infty)$\\
$J_{1}$&$(D_6\times D_{10},\infty), (7:6,\infty)$\\
$J_{2}$&$(A_5,\infty)$\\
$J_{2}.2$&$(S_5,\infty)$\\
$J_{3}$&$-$\\
$J_{3}.2$&$(19:18,3)$\\
$J_4$&$(M_{22}:2,2)$, $(11_+^{1+2}:(5\times 2S_4), 2)$,  $(L_2(32):5,11)$, $(L_2(23):2,2)$\\
& $(U_3(3),2)$, $(29:28,2)$, $(43:14,7)$, $(37:12,2)$\\
$^{2}F_4(2)'$&$-$\\
$^{2}F_4(2)$&$(M,2)$ where $M$ has order $156$\\
$Suz$&$(A_7,2)$, $(L_2(25),2)$\\
$Suz.2$&$(S_7,7)$\\
$McL$&$-$\\
$McL.2$&$-$\\
$HS$&$(M_{22},2)$\\
$HS.2$&$(M_{22}:2,2)$\\
$Co_3$&$(A_4\times S_5,?)$\\
$Co_2$&$(5_+^{1+2}:4S_4,2)$\\
$Co_1$&$(A_9\times S_3,3)$, $((A_7\times L_2(7)):2,2)$, $((D_{10}\times (A_5\times A_5).2).2,2)$\\
& $(5_+^{1+2}:\mathrm{GL}_2(5),2)$, $(5^3:(4\times A_5).2,2)$, $(5^2:4A_4,2)$, $(7^2:(3\times 2A_4),2)$\\
$He$&$(5^2:4A_4,2)$\\
$He.2$&$-$\\
$Fi_{22}$&$-$\\
$Fi_{22}.2$&$-$\\
$Fi_{23}$&$(L_2(23),2)$\\
$Fi_{24}'$&$((A_5\times A_9):2,3)$, $(A_6\times L_2(8):3,2)$, $(7:6\times A_7,7)$, $(U_3(3).2,2)$\\ &$(U_3(3).2,2)$, $(L_2(13).2,2)$, 
$(L_2(13).2,2)$, $(29:14,7)$\\
$Fi_{24}$&$(S_5\times S_9,3)$, $(S_6\times L_2(8):3,2)$, $(7:6\times S_7,7)$, $(7^{1+2}_+:(6\times S_3).2,2)$, $(29:28,7)$\\
$Ru$&$(L_2(13):2,2)$, $(5:4\times A_5,?)$, $(A_6.2^2,2)$, $(5_+^{1+2}:[2^5],2)$, $(3.A_6.2^2,2)$\\
$O'N$&$(A_7,2)$, $(A_7,2)$\\
$O'N.2$&$(31:30,5)$, $(L_2(7):2,2)$, $(A_6:2_2,2)$\\ 
$Ly$&$(67:22,11)$, $(37:18,3)$\\
$Th$&$(3^5:2S_6,2)$, $(5_{+}^{1+2}:4S_4,2)$, $(5^2:\mathrm{GL}_2(5),2)$,  $(7^2:(3\times 2S_4),2)$\\
& $(L_2(19).2,2)$, $(L_3(3),2)$, $(M_{10}=A_6.2_3,2)$, $(31:15,4)$, $(S_5,5)$\\
$HN$&$(3_+^{1+4}:4A_5,2)$\\
$HN.2$&$-$\\
$B$&$((2^2\times F_4(2)):2,?)$, $(3^{1+8}.2^{1+6}.U_4(2).2,?)$, 
$((3^2:D_8\times U_4(3).2^2).2,?)$, 
$(5:4\times HS:2,2)$ \\
&$(3^2.3^3.3^6.(S_4\times 2S_4),?)$, $(S_4\times {^2}F_4(2),2)$,  $(S_5\times (M_{22}:2),2)$, $((S_6\times (L_3(4):2)).2,2)$, 
\\&
$(5^3:L_3(5),2)$, $(5^{1+4}.2^{1+4}.A_5.4,2)$, $((S_6\times S_6).4,2)$,
 $((5^2:4S_4)\times S_5,2)$\\
& $(L_2(49).2,2)$, $(L_2(31),2)$, $(M_{11},2)$, $(L_3(3),2)$, $(L_2(17):2,2)$, $(L_2(11):2,2)$, $(47:23,23)$\\
\hline
\end{tabular}
\caption{Disposing of some of the sporadic simple groups.}\label{t: table1}
\end{table}

We have made use of the fact that full information on the maximal subgroups for each group in the first column of Table~$1$ is available: these are all stored in \texttt{GAP} or in~\texttt{magma}. To be more precise, in each case, either the maximal subgroup $M$ is stored providing a generating set (written as words in the standard generators for $G$), or when such information is not available (for instance, already for some of the maximal subgroups of $Fi_{23}$), the group $M$ is explicitly described (for instance, as a $p$-local subgroup) and hence also in this case it is possible to construct $M$ with a computer.

We are now able to prove Theorem~\ref{t: sporadic} for all groups bar the monster.

\begin{prop}\label{p: 1}
Let $G$ be an almost simple primitive group with socle a sporadic simple group. If $G$ is binary, then $G$ is the Monster group.
\end{prop}
\begin{proof}
In view of Table~\ref{t: table1}, it suffices to consider the case that $G$ is either $Co_3$, or $Ru$, or $B$: these are the only groups having a ``?'' symbol in one of their rows. We first assume that $G$ is either $Co_3$ or $Ru$; here, in view of Table~\ref{t: table1} the group $G$ is  acting on the cosets of $M=A_4\times S_5$ when $G=Co_3$, or $M=5:4\times A_5$ when $G=Ru$. Given a Sylow $2$-subgroup $P$ of $M$, in both cases it is easy to verify with \texttt{magma} that there exists $g\in \nor G P$ with $M\cap M^g=P$. When $G=Co_3$, $P$ is of type $2\times 2\times D_4$ and, when $G=Ru$, $P$ is of type $4\times 2\times 2$. Another computation shows that the actions of $A_4\times S_5$ on the cosets of $2\times 2\times D_4$, and of $5:4\times A_4$ on the cosets of $4\times 2\times 2$ are not binary. Therefore, $G$ in its action on the cosets of $M$ is not binary by Lemma~\ref{l: again0}.

Finally assume that $G$ is the Baby Monster $B$. In view of Table~\ref{t: table1}, $G$ is acting on the cosets of $M$ where $M$ is of one of the following types $$
(2^2\times F_4(2)):2, \hspace{1cm}
3^{1+8}.2^{1+6}.U_4(2).2, \hspace{1cm} 
(3^2:D_8\times U_4(3).2^2).2, \hspace{1cm}
3^2.3^3.3^6.(S_4\times 2S_4).$$ 
Let $\Omega$ be the set of right cosets of $M$ in $G$ and let $\alpha\in \Omega$ with $G_\alpha=M$ (that is, $\alpha$ is the coset $M$). We go through the four remaining cases one at a time.

{\bf  that $M\cong (3^2:D_8\times U_4(3).2^2).2$.} Observe that a Sylow $7$-subgroup of $G$ has order $7^2=49$, that $G$ has a unique conjugacy class of elements of order $7$, that $|M|$ and $|G:M|$ are both divisible by $7$. Then Lemma~\ref{l: M2} implies that $G$ is not binary.

\smallskip 

{\bf Suppose that $M\cong 3^{1+8}.2^{1+6}.U_4(2).2$.} The group $G$ has two conjugacy classes of elements of order $5$, with the ATLAS notation, of type 5A and of type 5B. By  computing  the  permutation  character  of $G$ via  the  package ``The GAP character Table Library" in GAP, we see that an element of type 5A fixes no point and that an element of type 5B fixes $25000$ points. Observe that $|M|$ is divisible by $5$, but not by $5^2=25$. Moreover, using the ATLAS~\cite{ATLAS}, we see that $G$ contains an elementary abelian $5$-group $V$ of order $5^3$ generated by three elements of type 5B; moreover, the normalizer of $V$ is a maximal subgroup of $G$ of type $5^3:L_3(5)$. In particular, each non-identity $5$-element of $V$ is of type 5B, because $L_3(5)$ acts transitively on the non-zero vectors of $5^3$. Since $|M|$ is not divisible by $25$, we conclude that $\Fix(V)$ is empty. Now we apply Lemma~\ref{l: added} to $L$, a subgroup of $M$ of order $25$ such that $|\Fix(L)|<|\Fix(g)|$. We conclude that $G$ is not binary.



\smallskip

{\bf Suppose that $M\cong 3^{2}.3^3.3^6.(S_4\times 2S_4)$.} From the ATLAS~\cite{ATLAS}, we see that $M=\nor G V$, where $V$ is an elementary abelian $3$-subgroup of order $3^2$ with $V\setminus\{1\}$  consisting only of elements of type 3B. For the proof of this case, we refer to~\cite{MR892191} and~\cite{MR1656568}. According to Wilson~\cite[Section~$3$]{MR892191}, there exist three $G$-conjugacy classes of elementary abelian $3$-subgroups of $G$ of order $3^2$ consisting only of elements of type 3B, denoted in~\cite{MR892191} as having type~(a), or~(b), or~(c). Moreover, from~\cite[Proposition~$3.1$]{MR892191}, we see that only for the elementary abelian $3$-groups of type~(a) the normalizers are maximal subgroups of $G$ and of shape  $3^{2}.3^3.3^6.(S_4\times 2S_4)$. Thus $V$ is of type~(a). Let $V_1,V_2,V_3$ be representatives for the conjugacy classes of elementary $3$-subgroups of $G$ of order $3^2$ and consisting only of elements of type 3B. We may assume that $V_1=V$. From~\cite{MR892191} or~\cite{MR1656568}, for $W\in \{V_1,V_2,V_3\}$, $\nor G W$ has shape $3^2.3^3.3^6.(S_4\times 2S_4)$, $(3^2\times 3^{1+4}).(2^2\times 2A_4).2$, and $(3^2\times 3^{1+4}).(2\times 2S_4)$; in \cite{MR892191,MR1656568}, these cases are refered to as type (a), type (b) and type (c), respectively.

Next, we consider a maximal subgroup $K$ of $G$ isomorphic to $\mathrm{PSL}_3(3)$. From \cite{MR1656568} (pages 9 and 10 and the discussion therein on the interaction between $K$ and the types~(a),~(b) and~(c)), we infer that $K$ contains a conjugate of $V$. In particular, replacing $K$ by a suitable $G$-conjugate, we may assume that $V\le K$. In particular, $M\cap K=\nor K V$. Take $\Lambda:=\alpha^K$ and observe that $\Lambda$ is  a $K$-orbit on $\Omega$ and that the stabilizer of the point $\alpha$ in $K$ is $\nor K V$. Moreover, since $K$ is maximal in $G$, we get $G_{\Lambda}=K$. We claim that $G^\Lambda=K^\Lambda$ is strongly non-binary, from which it follows that $G$ is not binary by~\ref{l: again12}. Observe that the action of  $K$ on $\Lambda$ is permutation isomorphic to the action of $K$ on the set of right cosets of $\nor K V$ in $K$.


Now, we consider the abstract group $K_0=\mathrm{PSL}_3(3)$, we consider an elementary abelian $3$-subgroup $V_0$ of order $9$ of $K_0$, we compute $N_0:=\nor {K_0}{V_0}$ and we consider the action of $K_0$ on the set $\Lambda_0$ of right cosets of $N_0$ in $K_0$. A straightforward computation shows that $K_0$ is not $2$-closed in this action, and hence $K_0$ in its action on $\Lambda_0$ is strongly non-binary by Lemma~\ref{l: fedup}.

\smallskip 

{\bf Suppose that $M\cong (2^2\times F_4(2)):2$.} Here we cannot invoke ``The GAP character Table Library" to understand whether $F_4(2)$ contains elements of type 5A or 5B, because the fusion of $M$ in $G$ is (in some cases) still unknown. As we mentioned above, $G$ has two conjugacy classes of elements of order $5$, denoted by 5A and 5B; what is more the group $F_4(2)$ contains a unique conjugacy class of elements of order $5$.  Observe that the centralizer in $G$ of an element of type 5A (respectively 5B) has order $44352000$ (respectively $6000000$). Now, the centralizer in $F_4(2)$ of a element of order $5$ has cardinality $3600$. Since $3600$ does not divide $6000000$, we get that $M$ contains only elements of type 5A; in particular elements of type 5B do not fix any element of $\Omega$.

Using \eqref{e: fora}, we conclude that if $g$ is an element of order $5$ in $M$, then
$$|\Fix_\Omega(g)|=\frac{|G|}{|M|}\frac{|M:\cent M g|}{|G:\cent G g|}=\frac{|\cent G g|}{|\cent M g|}=\frac{44352000}{3600\times 4 \times 2}=1540.$$

Now let $V$ be a Sylow $5$-subgroup of $M$ and observe that $V$ has order $5^2$ and $V\setminus\{1\}$ consists only of elements of type 5A. Referring to~\cite[Section~6]{MR892191}, we see that $G$ contains only one conjugacy class of elementary-abelian groups of order $25$ for which the non-trivial elements are all of type 5A. Thus $V$ is a representative of this $G$-conjugacy class. Now, Theorem~$6.4$ in~\cite{MR892191} yields $N_G(V)\cong 5^2:4S_4 \times S_5$. Appealing to \eqref{e: fora} again, we conclude that 
$$|\Fix_\Omega(V)|=\frac{|G|}{|M|}\frac{|M:\nor M V|}{|G:\nor G V|}=\frac{|\nor G V|}{|\nor M V|}=\frac{28800}{19200}=15.$$
Now Lemma~\ref{l: added} implies that $G$ is not binary.
\end{proof}

\section{The Monster}\label{s: monster}

In this section we prove Theorem~\ref{t: sporadic} for the monster. Our proof will break down into several parts, and to ensure that we cover all possibilities we will make use of a recent account of the classification of the maximal subgroups of the sporadic simple groups in~\cite{wilsonArXiv}.

From~\cite[Section~3.6]{wilsonArXiv}, we see that the classification of the maximal subgroups of the Monster $G$ is complete except for a few small open cases. In particular, if $M$ is a maximal subgroup of $G$, then either
\begin{description}
\item[(a)] $M$ is in~\cite[Section~$4$]{wilsonArXiv}, or
\item[(b)] $M$ is almost simple with socle isomorphic to $L_2(8)$, $L_2(13)$, $L_2(16)$, $U_3(4)$ or $U_3(8).$
\end{description} 

From here on $G$ will always denote the monster group, and $M$ will be a maximal subgroup of $G$. We consider the action of $G$ on cosets of $M$.

\subsection{The almost simple subgroups in {\bf (b)}}
We begin by applying Test 5 to those groups in category {\bf (b)}. Provided that such a group $M$ is not isomorphic to $L_2(16).2$, we find that, by applying Test 5 with the prime $2$ or $3$, we can immediately show that $G$ in its action on $G/M$ is not binary. 

The group $M=L_2(16).2$ is exceptional here: for each prime $p$ dividing $|M|$, there exists a permutation representation of $M$ of degree coprime to $p$ satisfying the four conditions in Lemma~\ref{l: alot}; hence we cannot apply Test~5. We defer the treatment of $L_2(16).2$ to \S\ref{s: snbs} below. 

From here on we will consider those groups in category {\bf (a)}, as well as the deferred group $L_2(16).2$.

\subsection{Constructing a strongly non-binary subset}\label{s: snbs}


Our next step is to apply Lemma~\ref{l: M2} to the remaining group, $L_2(16).2$, from category {\bf(b)} and to the groups from category {\bf (a)}. We start with a technical lemma; this is then followed by the statement that we need, Lemma~\ref{l: M3}.

\begin{lem}\label{l: M1}Let $G$ be the Monster, let $p\in \{5,7,11\}$ and let $x\in G$ with $o(x)=p$. Then there exists $g\in G$ with $\langle x,x^g\rangle$ elementary abelian of order $p^2$ and with $xx^g$ conjugate to $x$ via an element of $G$.
\end{lem}
\begin{proof}
When $p=11$, there is nothing to prove: $G$ has a unique conjugacy class of elements of order $11$ and a Sylow $11$-subgroup of $G$ is elementary abelian of order $11^2$.

When $p\in \{5,7\}$, it is enough to read~\cite[page 234]{ATLAS}: $G$ contains two conjugacy classes of elements of order $p$. Moreover, $G$ contains two elementary abelian $p$-subgroups $V$ and $V'$ both of order $p^2$, with $V$ generated by two elements of type pA and with $V'$ generated by two elements of type pB. Moreover, $\nor G V$ and $\nor G {V'}$ act transitively on the non-identity elements of $V$ and of $V'$, respectively.

This lemma can also be easily deduced from~\cite{wilsonoddlocal}.
\end{proof}

\begin{lem}\label{l: M3}
Let $G$ be the Monster and let $M$ be a maximal subgroup of $G$. If $M$ is as in the first column of Table~$\ref{t: table2'}$, then the action of $G$ on the right cosets of $M$ in $G$ is not binary.
\end{lem}

Note that the final line of the table is the remaining group from category {\bf (b)}, hence, once this lemma is disposed of, we only deal with groups from category {\bf (a)}.

\begin{proof}
If suffices to compare $|G:M|$ with $|M|$ and apply Lemmas~\ref{l: M1} and~\ref{l: M2}. For simplicity we have highlighted in the second column of Table~\ref{t: table2'} the prime $p$ that we are using to apply Lemma~\ref{l: M2}.
\end{proof}

\begin{table}[!ht]
  \begin{tabular}{|c|c|}\hline
    Maximal Subgroup &Prime\\\hline
$2.B$&$11$\\
$2^{1+24}.Co_1$&11\\
    $3.Fi_{24}$&$11$\\
    $2^{2}.{^{2}E_6(2)}.S_3$&11\\
    $2^{10+16}.O_{10}^+(2)$&7\\
    $2^{2+11+22}.(M_{24}\times S_3)$&7\\
    $3^{1+12}.2Suz.2$&7\\
    $2^{5+10+20}.(S_3\times L_5(2))$&7\\
    $2^{3+6+12+18}.(L_3(2)\times 3S_6)$&7\\
    $3^8.O_8^{-}(3).2_3$&7\\
    $(D_{10}\times HS).2$&7\\
    $(3^2:2\times O_8^+(3)).S_4$&7\\
    $3^{2+5+10}.(M_{11}\times 2S_4)$&11\\
    $5^{1+6}:2J_2:4$&$7$\\
$(A_5\times A_{12}):2$&7\\
$(A_5\times U_3(8):3_1):2$&7\\
$(L_3(2)\times S_4(4):2).2$&7\\
$(5^2:[2^4]\times U_5(5)).S_3$&7\\
$7^{1+4}:(3\times 2S_7)$&5\\\hline
$L_2(16).2$ & 5 \\\hline
  \end{tabular}
  \caption{Primitive actions of the Monster for Lemma~\ref{l: M3}.}\label{t: table2'}
  \end{table}

\subsection{Using Test 5}

We next apply Test 5 to the remaining maximal subgroups of $G$. The statement that we need is the following.
  
\begin{lem}\label{l: M4}
Let $G$ be the Monster and let $M$ be a maximal subgroup of $G$. If $M$ is as in the first column of Table~$\ref{t: table2}$, then the action of $G$ on the right cosets of $M$ in $G$ is not binary.
\end{lem}
\begin{proof}
Table~\ref{t: table2} lists precisely those remaining maximal subgroups that can be excluded using Test 5, together with the prime $p$ that has been used.
\end{proof}

\begin{table}[!ht]
\centering
\begin{threeparttable}
\begin{tabular}{@{}p{\textwidth}@{}}
\centering
\begin{tabular}{|c|c|c|}\hline
    Maximal Subgroup &Prime &Comments\\\hline
$(7:3 \times He):2$&$2$&\\
$(A_6\times A_6\times A_6).(2\times S_4)$&2&\\
    $(5^2:[2^4]\times U_3(5)).S_3$&$2$&\\
    $(L_2(11)\times M_{12}):2$&$2$&\\
    $(A_7\times (A_5\times A_5):2^2):2$&$2$&checked $4$-tuples\tnote{1}\\
    $5^4:(3\times 2L_2(25)):2$&$2$&\\
    $7^{2+1+2}:\mathrm{GL}_2(7)$&$2$&\\
    $M_{11}\times A_6.2^2$&$2$&\\
    $(S_5\times S_5\times S_5):S_3$&$3$&\\
    $(L_2(11)\times L_2(11)):4$&$2$&\\
    $13^2:(2L_2(13).4)$&$2$&\\
    $(7^2:(3\times 2A_4)\times L_2(7)).2$&$2$&\\
    $(13:6\times L_3(3)).2$&$2$&\\
    $13^{1+2}:(3\times 4S_4)$&$2$&\\
    $L_2(71)$&$2$&\\
    $L_2(59)$&$5$&\\
    $11^2:(5\times 2A_5)$&$2$&\\
$L_2(41)$&$2$&\\
    $L_2(29):2$&$2$&\\
    $7^2:\mathrm{SL}_2(7)$&$2$&\\
    $L_2(19):2$&$2$&\\
    $41:40$&$2$&\\\hline
  \end{tabular}\end{tabular}
  \caption{Primitive actions of the Monster for Lemma~\ref{l: M4}.}\label{t: table2}
  \begin{tablenotes}
   \item[1] \footnotesize This action turned out to be rather problematic. Each transitive action of odd degree satisfying the conditions~(1),~(2),~(3) in Lemma~\ref{l: alot} is not binary. However, for some of these actions to witness the non-binariness we had to resort to $4$-tuples, which was particularly time consuming.
  \end{tablenotes}
  \end{threeparttable}
  \end{table}
  
\subsection{The remainder}

By ruling out the groups listed in Tables~\ref{t: table2'} and \ref{t: table2}, we are left with precisely five subgroups on Wilson's list \cite{wilsonArXiv}. We now deal with these one at a time and, in so doing, we complete the proof of Theorem~\ref{t: sporadic}. The remaining groups are as follows:
$$S_3\times Th,\hspace{1cm} 
3^{3+2+6+6}:(L_3(3)\times SD_{16}),\hspace{1cm}
(7:3\times He):2,\hspace{1cm}
5^{3+3}.(2\times L_3(5)),\hspace{1cm}
5^{2+2+4}:(S_3\times \mathrm{GL}_2(5)).
$$

\smallskip

{\bf Suppose that $M\cong S_3\times Th$.} Here we refer to \cite[Section 2]{wilsonoddlocal}. There are three conjugacy classes of elements of order $3$ in the Monster $G$, of type 3A, 3B and 3C and, the normalizers of the cyclic subgroups generated by the elements of type 3C are maximal subgroups of $G$ conjugate to $M$. Choose $x$, an element of type 3C with $M=\nor G {\langle x\rangle }$. We write $M:=H\times K$, where $H\cong S_3$ and $K\cong Th$. From \cite[first two lines of the proof of Proposition~2.1]{wilsonoddlocal}, for every $y\in K$ of order $3$, $xy$ is an element of type 3C. From the subgroup structure of the Thompson group $Th$, the group $K$ contains an element $y$ of order $3$ with $\nor {K}{\langle y\rangle}$ of shape $(3\times G_2(3)):2$ and maximal in $K$. Since $x$ and $xy$ are in the same $G$-conjugacy class, there exists $g\in G$ with $x^g=xy$. Moreover, an easy computation inside the direct product $M=H\times K$ yields $M\cap M^g=\nor {G}{\langle x\rangle}\cap \nor G{\langle xy\rangle}=\nor M{\langle xy\rangle}\cong (\langle x\rangle\times \cent K y):2$ has shape $(3\times 3\times G_2(3)):2$. This shows that the action of $M$ on the right cosets of $M\cap M^g$ is permutation isomorphic to the primitive action of $Th$ on the right cosets of $(3\times G_2(3)):2$. In other words, $G$ has a suborbit inducing a primitive action of the sporadic Thompson group. From Proposition \ref{p: 1}, this action is not binary, and hence the action of $G$ on the right cosets of $M$ is not binary by Lemma~\ref{l: again0}.



\smallskip

{\bf Suppose that $M\cong 3^{3+2+6+6}:(L_3(3)\times SD_{16})$.} Arguing as in the previous case, we note that $M$ contains only elements of type 13A and no elements of type 13B. Let $Q$ be a $13$-Sylow subgroup of $M$ and let $P$ be a $13$-Sylow subgroup of $G$ with $Q\le P$. Observe that $P$ is an extraspecial group of exponent $13$ of order $13^3$ and that $Q$ has order $13$. Replacing $P$ by a suitable $G$-conjugate we may also assume that $Q\ne \Zent P$. (Observe that to guarantee that we may actually assume that $Q\ne \Zent P$ we need to use~\cite[page~15]{wilsonoddlocal}, which describes how the $13$-elements of type A and B are partitioned in $P$. Indeed, not all $13$-elements of type B are in $\Zent P$ and hence, if accidentally $Q=\Zent P$, we may replace $Q$ with a suitable conjugate.)  

Let $\alpha\in \Omega$ with $G_\alpha=M$ and set $\Lambda:=\alpha^P$. From the previous paragraph, $P$ acts faithfully on the set $\Lambda$ and $|\Lambda|=13^2$. Now the permutation group $P$ in its action on $\Lambda$ is not $2$-closed; indeed the $2$-closure of $P$ in its action on $\Lambda$ is of order $13^{14}$, it is a Sylow $13$-subgroup of $\Sym(\Lambda)$ (this follows from an easy computation or directly from~\cite{Dobson}). Since P embeds into $G^{\Lambda}$, the $2$-closure of $G^{\Lambda}$  contains the $2$-closure of $P$, but since $13^{14}$ does not divide the order of $|G|$, $G^{\Lambda}$ is not $2$-closed. Lemmas~\ref{l: again12} and \ref{l: fedup} imply that the action is not binary.

\smallskip
{\bf Suppose that $M\cong (7:3\times He):2$.} Observe that $He$ has a unique conjugacy class of elements of order $5$ and that its Sylow $5$-subgroups are elementary abelian of order $5^2$. Thus, we let $V:=\langle g,h\rangle$ be an elementary abelian $5$-subgroup of $M$ and we note that $g,h$ and $gh$ are $M$-conjugate and hence $G$-conjugate. The group $G$ has two conjugacy classes of elements of order $5$, denoted 5A and 5B. We claim that $M$ contains only elements of type 5A. Indeed, a computation inside the Held group $He$ reveals that $\cent M g$ contains an element of order $7\times 3\times 5=105$ and hence $G$ contains an element $x$ of order $105$ with $x^{21}=g$ being an element of order $5$. By considering the power information on the conjugacy classes of $G$, we see that $g$ belongs to the conjugacy class of type 5A. Since all $5$-elements are conjugate in $M$, we get that $M$ contains only $5$-elements of type 5A.

We now calculate the number of fixed points of $g$ and of $V$ on $\Omega$, making use of \eqref{e: fora}.
Using the information on the conjugacy classes of $He$ and $G$ we deduce that
$$|\Fix_\Omega(g)|=\frac{|G|}{|M|}\frac{|M:\cent M g|}{|G:\cent G g|}=\frac{|\cent G g|}{|\cent M g|}=\frac{1365154560000000}{12600}=108345600000.$$
Next, since $V$ is a Sylow $5$-subgroup of $M$, we deduce that $|\nor M V|=50400$ using the structure of the Held group. Moreover, from~\cite[Section~9]{wilsonoddlocal}, we get that the normalizer of an elementary abelian $5$-subgroup of the Monster consisting only of elements of type 5A is maximal in $G$ and is of the form $(5^2:4\cdot 2^2\times U_3(5)):S_3$. In particular, $|\nor G V|=302 400 000$. Thus
$$|\Fix_\Omega(V)|=\frac{|G|}{|M|}\frac{|M:\nor M V|}{|G:\nor G V|}=\frac{|\nor G V|}{|\nor M V|}=\frac{302400000}{50400}=6000.$$
Now Lemma~\ref{l: added} implies that $G$ is not binary.




\smallskip

{\bf Suppose that $M\cong 5^{3+3}.(2\times L_3(5))$.} Let $P$ be a Sylow $31$-subgroup of $M$ and observe that $P$ is also a Sylow $31$-subgroup of $G$. Recall that $G$ has a maximal subgroup $K:=C\times D$, where $C\cong S_3$ and $D\cong Th$ (as usual $Th$ denotes the sporadic Thompson group). Now, by considering the subgroup structure of $Th$, we see that $D$ contains a maximal subgroup isomorphic to $2^5.\mathrm{L}_5(2)$ and hence $D$ contains a Frobenius subgroup $F$ isomorphic to $2^5:31$. Replacing $F$ by a suitable conjugate we may assume that $P\le F$.

Comparing the subgroup structure of $M$ and of $F$, we deduce $M\cap F=P$. Consider $\Lambda:=\alpha^F$. By construction, as $M=G_\alpha$, we get $|\Lambda|=32$ and $F$ acts as a $2$-transitive Frobenius group of degree $32$ on $\Lambda$. Since the $2$-closure of a $2$-transitive group of degree $32$ is $\Sym(32)$ and since $G$ has no sections isomorphic to $\Sym(32)$, we deduce from Lemma~\ref{l: fedup} that $G^\Lambda$ is strongly non-binary. Therefore $G$ is not binary by Lemma~\ref{l: again12}.

\smallskip

{\bf Suppose that $M\cong 5^{2+2+4}:(S_3\times \mathrm{GL}_2(5))$.} For this last case we invoke again the help of a computer aided computation based on Lemma~\ref{l: alot}, but applied in a slightly different way than what we have described in Test~5. (We thank Tim Dokchitser for hosting the computations required for dealing with this case.) Observe that $|\Omega|-1$ is divisible by $5$, but not by $5^2$.

With \texttt{magma} we construct all the transitive permutation representations on a set $\Lambda$ of degree greater than $1$ and with $|\Lambda|$ not divisible by $5^2$ of $M$. (Considering that a Sylow $5$-subgroup of $M$ has index $576$, this computation does require some time but it is feasible.) Next, with a case-by-case analysis we see that none of these permutation representations satisfies (1),~(2),~(3) and~(4). Therefore, every transitive permutation representation of $M$ of degree greater than $1$ satisfing (1),~(2),~(3) and~(4) has degree divisible by $25$. Now, from Lemma~\ref{l: alot} applied with $d:=25$, we get that $G$ in its action on the set $M$ of right cosets of $M$ in $G$ is not binary because $25$ does not divide $|\Omega|-1$.

\end{document}